\documentclass[11pt]{amsart}

\usepackage[T1]{fontenc}
\usepackage[utf8]{inputenc}
\usepackage{lmodern}
\usepackage[american]{babel}
\usepackage[babel, final]{microtype}

\usepackage{amsmath,amssymb,amsthm,mathtools}
\usepackage[margin=1in]{geometry}
\usepackage[colorlinks=true,allcolors=blue]{hyperref}

\usepackage[numbers]{natbib}

\newcommand{\E}{\mathbb E}
\newcommand{\Pp}{\mathbb P}

\newcommand{\Hhalf}{H}      

\newcommand{\mue}{\mu^{\mathrm{even}}}
\newcommand{\muo}{\mu^{\mathrm{odd}}}

\newcommand{\meven}{m_{\mathrm{even}}}
\newcommand{\modd}{m_{\mathrm{odd}}}

\newcommand{\numset}[1]{\mathbb{#1}}

\newcommand{\Z}{\numset{Z}}

\newcommand{\cycle}[1]{\mathopen{(}#1\mathclose{)}}

\title{On the structure of the Poisson trinomial distribution}

\author[M. Broadie]{Mark Broadie}
\address{Graduate School of Business \\ Columbia University \\ New York, NY 10027}
\email{\href{mailto:mnb2@columbia.edu}{mnb2@columbia.edu}}
\urladdr{\url{https://business.columbia.edu/faculty/people/mark-broadie}}

\author[I. Petkova]{Ina Petkova}
\address{Department of Mathematics \\ Dartmouth College \\ Hanover, NH 03755}
\email{\href{mailto:ina.petkova@dartmouth.edu}{ina.petkova@dartmouth.edu}}
\urladdr{\url{https://math.dartmouth.edu/~ina/}}

\newtheorem{theorem}{Theorem}
\newtheorem{proposition}{Proposition}
\newtheorem{lemma}{Lemma}
\newtheorem{corollary}{Corollary}

\begin{document}

\begin{abstract}
We study sums of independent random variables that take values $0$, $1/2$, or $1$.
We show that the probability mass function of the sum splits into two interleaved parts: one supported on the integers and the other supported on the half-integers. Each part, when normalized, is a Poisson binomial distribution and hence log-concave with one or two modes. We also prove that each of the two conditional means (conditioning on being an integer or a half-integer) lies within $1/2$ of the unconditional mean. As a consequence, any two modes of the two conditional distributions are within $5/2$ of each other.
\end{abstract}

\maketitle

\section{Introduction}
The Poisson binomial distribution with parameters $n$ and $\{p_1, \ldots, p_n\}$ is the distribution of the sum of $n$ independent Bernoulli trials each with success probability $p_i$. It is known that for any choice of parameters the Poisson binomial distribution is log-concave with one or two modes. The closest integer to the mean is always a mode; the next closest integer to the mean may sometimes be a mode too. We generalize this to distributions whose support is $\{0, 1/2, 1\}$.

Suppose we have $n$ independent random variables, each taking values in $\{0, 1/2, 1\}$. We show that the distribution of their sum splits into two interleaved log-concave parts, supported on $\mathbb Z$ and on $\mathbb Z+\tfrac12$, with nearby conditional means and modes. To state the full result, we introduce some notation.

Let $X_1,\dots,X_n$ be independent random variables such that, for each $i$,
\[
\Pp(X_i=\tfrac12)=T_i,\qquad \Pp(X_i=1)=W_i,\qquad \Pp(X_i=0)=L_i\coloneqq1-T_i-W_i,
\]
and define $X = \sum_{i=1}^n X_i$. We call the distribution for $X$ a \emph{Poisson trinomial distribution with parameters $n$ and $\{(T_i, W_i)\}_{i=1}^n$}.

Define $\mu\coloneqq \E[X]$, and whenever $\Pp(X\in \mathbb Z)>0$ and $\Pp(X\in \mathbb Z+ \frac 1 2)>0$,  define
\[
\mue = \E\!\left[X\mid X\in \mathbb Z\right],
\qquad
\muo \coloneqq \E\!\left[X \mid X\in \mathbb Z+\tfrac12\right].
\]
The labels ``even'' and ``odd'' indicate whether an even or odd number of trials resulted in outcome $\frac 1 2$,  as we make explicit in Section~\ref{sec:means}.

Our main result describes the structure of the distribution when both parts are non-trivial: 

\begin{theorem}\label{thm:main}
Suppose $\Pp(X\in \mathbb Z)>0$ and $\Pp(X\in \mathbb Z+ \frac 1 2)>0$. 
Then the distributions for $X\mid (X\in \mathbb Z)$ and $X\mid (X\in \mathbb Z + \frac 1 2)$ are Poisson binomial, hence log-concave, with one or two modes. In particular, if $\meven$ and $\modd$ are modes  for $X\mid (X\in \mathbb Z)$ and $X\mid (X\in \mathbb Z + \frac 1 2)$, respectively, then 
\[
 |\meven-\mue|<1,\qquad |\modd-\muo|<1.
\]
Moreover,  
\[
|\mue - \mu |\leq \frac 1 2,\qquad
|\muo - \mu |\leq \frac 1 2, 
\]
and hence 
$|\mue - \muo| \leq 1$.

Similarly, 
\[
 |\meven-\mu|<\frac 3 2,\qquad |\modd-\mu|< \frac 3 2,
\]
and hence $|\meven - \modd | \leq \frac 5 2$. 
\end{theorem}

We also consider the degenerate case where one of the two parts of the distribution is trivial:

\begin{theorem}\label{thm:A}
Suppose  $\Pp(X\in \mathbb Z) = 0$ or $\Pp(X\in \mathbb Z+ \frac 1 2) = 0$.  Then $T_i\in \{0,1\}$ for $i = 1, \ldots, n$, and if $k$ is the number of indices $i$ for which  $T_i = 0$, then the distribution is the Poisson binomial on the respective $k$ trials shifted to the right by $\frac{n-k}{2}$.
\end{theorem}

The proof of Theorem~\ref{thm:main} proceeds in two stages. In Section~\ref{sec:means}, we compute the three means explicitly and show the bound claimed in Theorem~\ref{thm:main}. In Section~\ref{sec:modes}, we prove unimodality, and in fact log-concavity of the two conditional distributions, and relate the modes to the means, to complete the proof of Theorem~\ref{thm:main}. We discuss the degenerate case in Section~\ref{sec:deg}. 

The Poisson trinomial distribution arises naturally in several applied settings. Perhaps the most direct motivation is team competitions with win-tie-loss formats, such as the Ryder Cup in golf, the Davis Cup in tennis, or international soccer  tournaments, where each one-on-one match contributes $1$, $\frac{1}{2}$, or $0$ points to the team score and the total score is a sum of independent $\{0,\frac{1}{2},1\}$-valued random variables. A potential source of examples is reliability theory, where systems with three states (failed, degraded, fully 
functioning) and independent components lead to the same distributional structure. Another potential source arises in clinical trials with ordered categorical outcomes (worse, unchanged, improved), which are naturally coded as $0$, $\frac{1}{2}$, and $1$. In each of these settings, the results of this paper---in particular the log-concavity and the proximity of the conditional means and modes to the 
unconditional mean---provide useful structural information about the distribution 
of the aggregate outcome. We elaborate on the team competition application in Section~\ref{sec:apps}, deriving a couple of optimization results.

\subsection*{Acknowledgments} We thank Pete Winkler for helpful comments. IP was partially supported by NSF CAREER Grant DMS-2145090.

\section{Means}\label{sec:means}

In this section, we assume $\Pp(X\in \mathbb Z)>0$ and $\Pp(X\in \mathbb Z+ \frac 1 2)>0$, and we
 study the means of the distributions for $X$,  $X\mid (X\in \mathbb Z)$, and $X\mid (X\in \mathbb Z + \frac 1 2)$. We  show the following: 
\begin{proposition}\label{prop:bound}
Each of the conditional means is within $\frac 1 2 $ of the mean for $X$:
\[
|\mue - \mu |\leq \frac 1 2,\qquad
|\muo - \mu |\leq \frac 1 2.
\]
\end{proposition}

To prove Proposition~\ref{prop:bound}, we first  make some explicit computations. 

Define indicator variables and their sum
\[
 Z_i \coloneqq \mathbf 1_{\{X_i=\tfrac12\}},\qquad S\coloneqq \sum_{i=1}^n Z_i
\]
and observe that, in terms of $S$, the conditional means become
\[
\mue=\E[X\mid S \text{ even}],\qquad
\muo=\E[X\mid S \text{ odd}].
\]
Indeed, $2X=\sum_{i=1}^n (2X_i)$ is an integer and $2X\equiv \sum_{i=1}^n Z_i = S\pmod 2$, so $X\in\mathbb Z$ if and only if $S$ is even, and $X\in\mathbb Z+\tfrac12$ if and only if $S$ is odd. 

To facilitate the exposition, also define
\[
a \coloneqq \E[(-1)^S],\qquad b \coloneqq \E[X(-1)^S]. 
\]
Then
\[
\Pp(S\text{ even})=\E\!\left[\frac{1+(-1)^S}{2}\right]=\frac{1+a}{2},
\qquad
\E[X\mathbf 1_{\{S\text{ even}\}}]=\E\!\left[X\frac{1+(-1)^S}{2}\right]=\frac{\mu+b}{2},
\]
and hence
\[
\mue = \frac{\mu+b}{1+a}.
\]
Similarly, we obtain 
\[
\muo = \frac{\mu-b}{1-a}.
\]
The differences of means that we are interested in become 
\begin{equation}\label{eq:mu-diff}
\mue-\mu=\frac{b-a\mu}{1+a}, \qquad \muo-\mu = - \frac{b-a\mu}{1-a}.
\end{equation}

We next compute $a$, $b$, and $\mu$ explicitly. 
Since the $Z_j$ are independent,
\[
a=\E\!\left[\prod_{j=1}^n (-1)^{Z_j}\right]=\prod_{j=1}^n \E[(-1)^{Z_j}].
\]
For each $j$, $Z_j\sim\mathrm{Bernoulli}(T_j)$, so
\[
\E[(-1)^{Z_j}]=(1-T_j)\cdot 1 + T_j\cdot(-1)=1-2T_j.
\]
Hence
\begin{equation}\label{eq:a_product_local}
a=\prod_{j=1}^n (1-2T_j).
\end{equation}
Next,
\begin{align*}
b
&=\E\!\left[X(-1)^S\right]
=\E\!\left[\left(\sum_{i=1}^n X_i\right)\prod_{j=1}^n (-1)^{Z_j}\right]
=\sum_{i=1}^n \E\!\left[X_i\prod_{j=1}^n (-1)^{Z_j}\right]\\
&=\sum_{i=1}^n \E\!\left[X_i(-1)^{Z_i}\prod_{j\ne i} (-1)^{Z_j}\right].
\end{align*}
By independence across indices, $(X_i,Z_i)$ is independent of $\{Z_j\}_{j\ne i}$, hence
\[
\E\!\left[X_i(-1)^{Z_i}\prod_{j\ne i} (-1)^{Z_j}\right]
=\E\!\left[X_i(-1)^{Z_i}\right]\prod_{j\ne i}\E[(-1)^{Z_j}].
\]
We already have $\E[(-1)^{Z_j}]=1-2T_j$. Also,
\[
\E[X_i(-1)^{Z_i}]
=1\cdot W_i\cdot 1+\left(\tfrac12\right)\cdot T_i\cdot(-1)+0\cdot L_i\cdot 1
= W_i-\frac{T_i}{2}.
\]
Therefore
\begin{equation}\label{eq:b_local}
b=\sum_{i=1}^n\left(W_i-\frac{T_i}{2}\right)\prod_{j\ne i}(1-2T_j).
\end{equation}
Last, 
\begin{equation}\label{eq:mu}
\mu=\E[X]=\sum_{i=1}^n \E[X_i]=\sum_{i=1}^n\left(W_i+\frac{T_i}{2}\right).
\end{equation}

We can now compute the numerator that appears in \eqref{eq:mu-diff}:
\begin{lemma}\label{lem:b-am}
For $a,b,\mu$ as above, we have
\begin{equation}\label{eq:b-am}
b-a\mu=\sum_{i=1}^n T_i(W_i-L_i)\prod_{j\ne i}(1-2T_j),
\qquad \textrm{where } L_i=1-T_i-W_i.
\end{equation}
\end{lemma}
\begin{proof}
From \eqref{eq:a_product_local} and \eqref{eq:mu}, we get
\begin{align*}
a\mu
&=\left(\prod_{j=1}^n(1-2T_j)\right)\sum_{i=1}^n\left(W_i+\frac{T_i}{2}\right)\notag\\
&=\sum_{i=1}^n\left(W_i+\frac{T_i}{2}\right)(1-2T_i)\prod_{j\ne i}(1-2T_j).
\end{align*}
Combining with \eqref{eq:b_local}, we have
\begin{align*}
b-a\mu
&=\sum_{i=1}^n
\Biggl[
\left(W_i-\frac{T_i}{2}\right)
-\left(W_i+\frac{T_i}{2}\right)(1-2T_i)
\Biggr]\prod_{j\ne i}(1-2T_j).
\end{align*}
Simplifying results in \eqref{eq:b-am}.
\end{proof}

Next, we derive a bound on  $b-a\mu$:

\begin{lemma}\label{lem:key}
For $a,b,\mu$ as above, we have
\[
|b-a\mu|\le \frac{1-|a|}{2}.
\]
\end{lemma}

\begin{proof}
Let $\alpha_i \coloneqq |1-2T_i|\in[0,1]$ and $c \coloneqq \prod_{i=1}^n\alpha_i$. Note that $c = |a|$.

Taking absolute values in Equation~\eqref{eq:b-am} and using  $\alpha_j=|1-2T_j|$ gives
\begin{equation}\label{eq:abs1}
|b-a\mu|\le \sum_{i=1}^n T_i\,|W_i-L_i|\prod_{j\ne i}\alpha_j.
\end{equation}
Since $W_i,L_i\ge 0$, we have $|W_i-L_i|\le W_i+L_i = 1 - T_i$, so 
\begin{equation}\label{eq:pointwise}
T_i\,|W_i-L_i|\le  T_i (1 - T_i)\le  \min(T_i,1-T_i)=\frac{1-\alpha_i}{2}.
\end{equation}
Substituting \eqref{eq:pointwise} into \eqref{eq:abs1},
\begin{equation}\label{eq:abs2}
|b-a\mu|\le \sum_{i=1}^n \frac{1-\alpha_i}{2}\prod_{j\ne i}\alpha_j.
\end{equation}
Define partial products $r_k\coloneqq\prod_{j=k}^n\alpha_j$ and $r_{n+1}\coloneqq1$. Then
\[
1-r_1=\sum_{k=1}^n (r_{k+1}-r_k)=\sum_{k=1}^n (1-\alpha_k)r_{k+1}.
\]
Also $\prod_{j\ne k}\alpha_j=\left(\prod_{j<k}\alpha_j\right)r_{k+1}\le r_{k+1}$, hence
\[
(1-\alpha_k)\prod_{j\ne k}\alpha_j \le (1-\alpha_k)r_{k+1}.
\]
Summing,
\[
\sum_{k=1}^n (1-\alpha_k)\prod_{j\ne k}\alpha_j \le \sum_{k=1}^n (1-\alpha_k)r_{k+1} = 1-r_1=1-c.
\]
Combining with \eqref{eq:abs2}, we conclude
\[
|b-a\mu|\le \frac{1-c}{2}=\frac{1-|a|}{2}.\qedhere
\]
\end{proof}

We can now prove Proposition~\ref{prop:bound}: 
\begin{proof}[Proof of Proposition~\ref{prop:bound}]
Since  $\Pp(S\text{ even})>0$ and $\Pp(S\text{ odd})>0$, we have $1+a>0$ and $1-a>0$, and hence $1 + |a|>0$. We also have the inequalities $1-|a|\le 1+a$ and $1-|a|\le 1-a$, which follow from $|a|\ge \pm a$, so by Lemma~\ref{lem:key} we get
\[
|\mue-\mu| = \frac{|b-a\mu|}{1+a}
\le \frac{\frac{1-|a|}{2}}{1+a}
\le \frac{1}{2}
\]
and
\[
|\muo-\mu| = \frac{|b-a\mu|}{1-a}
\le \frac{\frac{1-|a|}{2}}{1-a}
\le \frac{1}{2}.\qedhere
\]
\end{proof}

\section{Log-concavity and modes}\label{sec:modes}

In this section, we show that the two parity-conditioned distributions are log-concave (and hence unimodal). 

\subsection{Generating functions and parity extraction}
Define the integer-valued random variable $H = 2X$, and 
observe that the pgf for $H$ is
\[
G(w):=\E[w^{\Hhalf}]=\prod_{i=1}^n(L_i+T_i w+W_iw^2).
\]
Write
\[
G(w)=\sum_{m=0}^{2n} a_m w^m,\qquad a_m=\Pp(\Hhalf=m).
\]
Define the even- and odd-part polynomials $p$ and $q$ by
\[
G(w)=p(w^2)+w\,q(w^2),
\]
i.e.
\[
p(z)=\sum_{k=0}^n a_{2k}z^k,\qquad q(z)=\sum_{k=0}^{n-1} a_{2k+1}z^k.
\]
Normalizing, we obtain the conditional pmfs
\begin{equation}
\label{eq:evenpmf}
\Pp(\Hhalf=2k\mid S\text{ even})=\frac{[z^k]\,p(z)}{p(1)},\qquad
\Pp(\Hhalf=2k+1\mid S\text{ odd})=\frac{[z^k]\,q(z)}{q(1)}.
\end{equation}

Recall that a real polynomial $f(w)\in\mathbb{R}[w]$ is called Hurwitz stable if all its roots lie in the open left half-plane
$\{z\in\mathbb{C}:\mathrm{Re}(z)<0\}$.

\begin{lemma}
\label{lem:quadratic-hurwitz}
If $L,T,W>0$, then $L+Tw+Ww^2$ is Hurwitz stable.
\end{lemma}

\begin{proof}
The roots are $\frac{-T \pm \sqrt{T^2 - 4LW}}{2W}$. If $T^2 < 4LW$, the roots are complex with real part $-T/(2W) < 0$. If $T^2 \ge 4LW$, the roots are real and both negative, since $\sqrt{T^2 - 4LW} \le \sqrt{T^2} = T$ implies both $-T + \sqrt{T^2-4LW} \le 0$ and $-T - \sqrt{T^2-4LW} \le 0$. In either case all roots lie in the open left half-plane, which can also be seen as an instance of the Routh--Hurwitz criterion for degree~$2$ polynomials.
\end{proof}

\begin{lemma}
\label{lem:product-hurwitz}
If $f_1,\dots,f_n$ are Hurwitz stable real polynomials, then $\prod_{i=1}^n f_i$ is Hurwitz stable.
\end{lemma}

\begin{proof}
The roots of the product are the multiset union of roots of the factors.
\end{proof}

By Lemmas \ref{lem:quadratic-hurwitz} and \ref{lem:product-hurwitz}, $G(w)$ is Hurwitz stable and satisfies $G(0)=\prod_i L_i>0$.

\begin{theorem}[cf.\ \cite{holtzHB} Theorem 1]
\label{thm:HB}
Let $f(w)\in\mathbb{R}[w]$ be Hurwitz stable and write
\[
f(w)=g(w^2)+w\,h(w^2)
\]
with $g,h\in\mathbb{R}[w]$. If $f(0)\neq 0$, then all zeros of $g$ and of $h$ are real and non-positive.
\end{theorem}

\begin{proof}
This is a version of the Hermite--Biehler Theorem; see \cite{gantmacher} or \cite{holtzHB}.
\end{proof}

Returning to our conditional distributions:

\begin{proposition}
\label{cor:unimodal}
The conditional distributions $\Hhalf\mid(S\text{ even})$ and $\Hhalf\mid(S\text{ odd})$ are log-concave (and hence unimodal) on their natural lattices.
Equivalently, $X\mid(S\text{ even})$ is log-concave on $\{0,1,\dots,n\}$ and
$X\mid(S\text{ odd})$ is log-concave on $\{\tfrac12,\tfrac32,\dots,n-\tfrac12\}$. 
\end{proposition}

\begin{proof}
By Theorem~\ref{thm:HB} applied to $f = G$, the polynomials $p$ and $q$ have only real, non-positive zeros. Since their coefficients are non-negative, 
for each polynomial they form a log-concave and hence unimodal sequence. This is a standard consequence of Newton's inequalities; see, e.g., \cite{hardy}. Normalization by the constants $p(1)$ and $q(1)$  does not affect log-concavity or unimodality.
\end{proof}

\subsection{Poisson binomial representation and mode--mean proximity}

Since $p$ is a polynomial with positive coefficients and non-positive real roots, it can be written as
\[
p(z)=p(0)\prod_{j=1}^n (\beta_j+ z),\qquad \beta_j\geq 0,
\]
and similarly for $q$. After normalizing,  both $p$ and $q$ are pgfs of sums of independent Bernoulli random variables, i.e., Poisson binomial distributions.
The same holds for $q$.

\begin{theorem}[cf.\ \cite{darroch} Theorem 2 and Theorem 4]
\label{thm:darroch}
Let $B=\sum_{j=1}^m B_j$ be a sum of independent Bernoulli random variables with mean $\E[B]=\mu_B$.
If $m_B$ is any mode of $B$ (if there are two, they are adjacent and either may be chosen), then
\[
|m_B-\mu_B|<1.
\]
\end{theorem}

\begin{proof}
See \cite{darroch}, for example.
\end{proof}

We now look at the conditional modes. Let
\[
\meven\in\arg\max_{k\in\{0,1,\dots,n\}} \Pp(X=k \mid S\text{ even}),
\qquad
\modd\in\arg\max_{k\in\{\tfrac12,\tfrac32,\dots,n-\tfrac12\}} \Pp(X=k \mid S\text{ odd}).
\]

\begin{corollary}
\label{cor:mode-mean}
The conditional distributions
$X\mid( S\text{ even})$ and $X\mid(S\text{ odd})$ each have either one mode or two adjacent modes, and 
\[
 |\meven-\mue|<1,\qquad |\modd-\muo|<1.
\]
\end{corollary}

\begin{proof}
In the even case, the integer-valued random variable
\[
X_{\mathrm{even}}:=\frac{\Hhalf}{2}\ \Big|\ (S\text{ even})
\]
has pgf $p(z)/p(1)$ by \eqref{eq:evenpmf}, hence is Poisson--binomial. Theorem~\ref{thm:darroch} gives
$|\meven-\E[X_{\mathrm{even}}]|<1$. But $\E[X_{\mathrm{even}}]=\E[X\mid S\text{ even}]=\mue$.

In the odd case, the integer-valued
\[
X_{\mathrm{odd}}:=\frac{\Hhalf-1}{2}\ \Big|\ (S\text{ odd})
\]
has pgf $q(z)/q(1)$ and is Poisson--binomial.  Theorem~\ref{thm:darroch} gives
$|\text{mode}(X_{\mathrm{odd}})-\E[X_{\mathrm{odd}}]|<1$, which translates to
$|\modd-\muo|<1$.
\end{proof}

Combining Corollary~\ref{cor:mode-mean} with Proposition~\ref{prop:bound}, we can put a universal bound on both types of modes, in terms of distance to the mean $\mu$:

\begin{proposition}
\label{prop:mode-mean}
If $\meven$ and $\modd$ are modes of the two conditional distributions $X\mid( S\text{ even})$ and $X\mid(S\text{ odd})$, respectively, then 
\[
 |\meven-\mu|<\frac 3 2,\qquad |\modd-\mu|< \frac 3 2.
\]
In particular, $|\meven - \modd | \leq \frac 5 2$. 
\end{proposition}
\begin{proof}
The first part is an immediate consequence of Corollary~\ref{cor:mode-mean} and Proposition~\ref{prop:bound}. The second part follows since the first part implies $|\meven - \modd | < 3$, and the two modes live on the integer and half-integer lattice. 
\end{proof}

\begin{proof}[Proof of Theorem~\ref{thm:main}]
The theorem is immediate from Proposition~\ref{prop:bound}, Theorem~\ref{cor:unimodal}, and Proposition~\ref{prop:mode-mean}.
\end{proof}

\section{The degenerate case}\label{sec:deg}

For completeness, we also consider the degenerate case where  $\Pp(X\in \mathbb Z) = 0$ or $\Pp(X\in \mathbb Z+ \frac 1 2) = 0$. 

\begin{proof}[Proof of Theorem~\ref{thm:A}]
Suppose there is some $T_i\in (0,1)$. Since the probabilities $\Pp\left(X-X_i\in \mathbb Z+\frac 1 2\right)$ and $\Pp\left(X-X_i\in \mathbb Z\right)$ add up to one, at least one of them is positive. Then 
\[
\Pp(X\in \mathbb Z)  = T_i \, \Pp\left(X-X_i\in \mathbb Z+\frac 1 2\right) + (1-T_i)\, \Pp(X-X_i\in \mathbb Z)
\]
is positive too. Similarly, $\Pp(X\in \mathbb Z +\tfrac12)$ is positive. This contradicts the assumption. Thus, $T_i\in \{0,1\}$ for all $i$. 

By reordering, assume $T_1 = \cdots = T_k = 0$ and $T_{k+1} = \cdots = T_n = 1$. The distribution for $X_1+ \cdots+ X_k$ is Poisson binomial, while $X_{k+1}+ \cdots+ X_n$ has a point mass distribution at $\frac{n-k}{2}$. The distribution for the sum $X = (X_1+\cdots+X_k) + (X_{k+1}+\cdots+X_n)$ is therefore a Poisson binomial distribution shifted by $\frac{n-k}{2}$.
\end{proof}

Note that in this degenerate case, the mean is $\mu = \frac{n-k}{2} + \sum_{i=1}^k W_i$, and any mode is within distance $1$ of $\mu$ by \cite{darroch}.

\section{Applications}\label{sec:apps}

In \cite{BroadiePetkova2026}, we apply Theorem~\ref{thm:main} to study optimal player matchups in team competitions with win-tie-loss formats. Specifically, we consider series of one-on-one games in which one team enters with a score deficit or a big lead and seeks to win the series overall.

An example of such an event is the last day of the Ryder Cup competition in golf. Each team has 12 players and the competition consists of three days of match play worth $28$ points in total.  The final day features 12 singles matches, where each team member plays one-on-one against an opponent from the other team. The team with the most total points wins the cup; in the event of a tie, the defending champion retains the cup.

Consider, most generally, a match-play competition where team~$A$ with players of strengths $a_1 \geq a_2 \geq \cdots \geq a_n$ faces team~$B$ with players of strengths $b_1 \geq b_2 \geq \cdots \geq b_n$. An \emph{ordering} $\sigma$ assigns team~$B$'s players to positions, so player $\sigma(i)$ on team~$B$ faces player $i$ on team~$A$ in match $i$. Denote the random number of points ($1$ point to the winner, or half a point to each team in case of a tie) earned by team~$B$ in match $i$  by $X_{\sigma(i),i}$.  We take the perspective of team~$B$.  Suppose team $A$ has ordered its players, and suppose team $B$ needs at least $k$ points in the $n$ matches in order to win the competition, where $k\in \{$0, 0.5, $\ldots$, $n-0.5$, $n\}$ depends on the number of points accumulated in previous events. 
The goal for team~$B$ is to find an ordering $\sigma$ that maximizes 
\begin{equation}
\label{eq:obj}
P(X_{\sigma} \geq k),
\end{equation}
where
\[
X_{\sigma} = X_{\sigma(1),1} + \cdots + X_{\sigma(n),n}
\]
is the total random points earned. Note that $X_\sigma$ follows a Poisson trinomial distribution.

In many scenarios, the probability that a player wins a match against another is best modeled by a logistic regression in terms of the difference in strength; tie and loss probabilities are then derived using symmetries. When the strength range of the competition pool is not very wide, e.g.\ in professional sports, this model is very well approximated by a linear one. In a linear model, the probabilities that a player wins, loses, or ties a match against another are modeled by the linear functions
\begin{align*}
W(s) & = \alpha s + \beta, \\
L(s) & = - \alpha s + \beta,\\
T(s) & = 1-2\beta, 
\end{align*}
where $s$ is the strength differential, and $\alpha> 0$ and $\beta\in [0,1/2]$ are constants. A necessary assumption is that the domain for the three functions, i.e., the set of all possible values of the strength differential, is contained in the interval $(-\frac{\beta}{\alpha}, \frac{\beta}{\alpha})\cap (\frac{\beta -1}{\alpha}, \frac{1-\beta}{\alpha})$; this ensures $0\leq W(s)\leq 1$. 

Given such a linear model, we solve this optimization problem for almost all values of $k$. 

First observe that the expected total score $\E[X_\sigma]$ is independent of $\sigma$: 
 
 \begin{proposition}\label{prop:lin-E}
    For any ordering $\sigma$, 
    \[\E[X_{\sigma}] = \frac{n}{2} + \alpha\sum_{i=1}^n b_i -\alpha\sum_{i=1}^n a_i.\]
\end{proposition}

\begin{proof}    
    Fix an ordering $\sigma$. Using linearity of expectation and rearranging indices, we get
    \begin{align*}
    \E[X_{\sigma}] &= \sum_{i=1}^n\E[X_{\sigma(i),i}]\\
    &=\sum_{i=0}^n\left( \left( \alpha \left(b_{\sigma(i)} - a_i\right)  + \beta\right) + \frac{1}{2}\left(1-2\beta\right) \right)\\
    &=  \frac{n}{2} + \alpha\sum_{i=1}^n b_i -\alpha\sum_{i=1}^n a_i. \qedhere
    \end{align*}
\end{proof}
Since $E[X_{\sigma}]$ only depends on the teams $A$ and $B$ but not on the specific ordering, from here on we denote this expectation by $\mu$.

We show that the tail $P(X_{\sigma}\geq k)$ is maximized by the strong-vs-strong order for sufficiently large $k$ and by strong-vs-weak order for sufficiently small $k$, and that in fact that are at most eight values of $k\in \frac 1 2 \Z$ for which the tail is maximized by a different order:

\begin{theorem}\label{thm:linear_ties}
Assume a linear model with tie probability $T(s) = 1 - 2\beta$ for $\beta \in (0, \frac{1}{2})$.
If $k \geq \mu + 2.5$, then $\Pp(X_\sigma \geq k)$ is maximized by $\sigma = (1, 2, \ldots, n)$ (strong-vs-strong).
If $k \leq \mu - 2$, then $\Pp(X_\sigma \geq k)$ is maximized by $\sigma = (n, n-1, \ldots, 1)$ (strong-vs-weak).
For $k\in (\mu-2, \mu+2.5)$, the optimal ordering may depend on how the teams $A$ and $B$ are chosen.
\end{theorem}
 
\begin{proof}
The proof of the theorem relies on an analysis of how a swap of two players in an ordering affects the probability of accumulating enough points. 

Fix an ordering $\sigma$, and transposition $\tau^{ij} = \cycle{i, j}$ with $i<j$.  A direct computation shows that 
\[
\Pp(X_{\sigma(i), j} + X_{\sigma(j), i} = s) - \Pp(X_{\sigma(i), i} + X_{\sigma(j), j} = s)  = 
\begin{cases}
    \Delta & \text{if } s = 0, 2 \\
   -  2\Delta  & \text{if } s = 1 \\
    0 & \text{otherwise.}
\end{cases}
\]
where
 \[\Delta  = -\alpha^2(a_i - a_j)(b_{\sigma(i)} - b_{\sigma(j)}).\]
Write $Y = X_{\sigma} - X_{\sigma(i),i} - X_{\sigma(j),j}$. Another direct computation shows that 
    \[
    \Pp(X_{\sigma\tau^{ij}}\geq k ) - \Pp(X_{\sigma}\geq k) =  \Delta \cdot f(Y, k), 
    \]
    where 
    \[f(Y, k) = (\Pp(Y=k-2)-  \Pp(Y=k-1)) + ( \Pp(Y=k-1.5) -  \Pp(Y=k-0.5)).\]
Thus, the difference $ \Pp(X_{\sigma\tau^{ij}}\geq k ) - \Pp(X_{\sigma}\geq k)$ is positive if and only $\sigma(i) > \sigma(j)$ and $f(Y, k)>0$, or $\sigma(i) < \sigma(j)$ and $f(Y, k)<0$. 

Note that Theorem~\ref{thm:main} implies that the distributions for $Y\mid (Y\in \mathbb Z)$ and $Y\mid (Y\in \mathbb Z + \frac 1 2)$ are both decreasing to the right of $\E[Y] + 0.5$. Since $\mu > \E[Y]$, it follows that $\Pp(Y = t) - \Pp(Y=t+1) > 0$ whenever $t\ge \mu+0.5$. In particular, $f(Y,k)$ is positive whenever $k\ge \mu + 2.5$. 
Thus, performing a swap that eliminates an inversion increases the probability of winning at least $k$ matches, for $k\ge \mu + 2.5$. Since we can get from any ordering to $(1, 2, \ldots, n)$ by sequentially eliminating inversions, $\sigma = (1, 2, \ldots, n)$ maximizes $P(X_{\sigma}\geq k)$. 

Similarly, $\Pp(Y = t) - \Pp(Y=t+1) < 0$ whenever $t+1\le \mu-2.5$, so $f(Y,k)$ is negative whenever $k \le \mu -2$. 
Thus, creating a new inversion increases the probability of winning at least $k$ matches, hence this probability is maximized for the ordering $(n, \ldots, 2, 1)$.
\end{proof}

In the degenerate case when there is zero probability of a tie, i.e., $\beta = \frac 1 2$, we solve the ordering problem for all but two integer values of $k$:

\begin{theorem}\label{thm:linear_no_ties}
Assume a linear model with tie probability $T(s) = 0$.
 If $k \geq \mu + 2$, then $\Pp(X_\sigma \geq k)$ is maximized by $\sigma = (1, 2, \ldots, n)$.
 If $k \leq \mu - 1$, then $\Pp(X_\sigma \geq k)$ is maximized by $\sigma = (n, n-1, \ldots, 1)$.
For $k\in (\mu-1, \mu+2)$, the optimal ordering may depend on how the teams $A$ and $B$ are chosen.
\end{theorem}
\begin{proof}
The proof is a simpler analogue of the proof of Theorem~\ref{thm:linear_ties}.
\end{proof}

\bibliographystyle{mwamsalphack}
\bibliography{references}

\end{document}